\documentclass[10pt]{article}
\usepackage{a4}
\usepackage[margin=1.25in]{geometry}
\usepackage{amsthm, amsmath, amsfonts, cite}
\usepackage{graphicx}
\usepackage[margin=0.8cm]{caption}
\usepackage{titlesec}
\graphicspath{ {./Paperfigures/} }
\usepackage{epsfig}
\usepackage{caption}  
\usepackage{capt-of}
\usepackage{tikz}
\usepackage{multicol}
\usepackage{thmtools}
\usepackage{thm-restate}
\usepackage{cleveref}
\usepackage{enumerate}
\usepackage{braket}
\usetikzlibrary{angles, quotes}
\newtheorem{theorem}{Theorem}

\newtheorem{lemma}[theorem]{Lemma}

\allowdisplaybreaks

\title{Improved bounds for colouring circle graphs}
\author{James Davies\thanks{Department of Combinatorics and Optimization, University of Waterloo, Waterloo, Canada. E-mail: \texttt{jgdavies@uwaterloo.ca}.}}
\date{}

\begin{document}

\maketitle

\begin{abstract}
	We prove the first $\chi$-bounding function for circle graphs that is optimal up to a constant factor. To be more precise, we prove that every circle graph with clique number at most $\omega$ has chromatic number at most $2\omega \log_2 (\omega) +2\omega \log_2(\log_2 (\omega)) + 10\omega$.
\end{abstract}

\section{Introduction}

A \emph{circle graph} is an intersection graph of chords on a circle; each vertex corresponds to a chord and two vertices are adjacent whenever their corresponding chords intersect. A class of graphs is \emph{$\chi$-bounded} if there exists some function which bounds the maximum chromatic number of graphs in the class in terms of their clique number $\omega$.
We call such a function a \emph{$\chi$-bounding function}.

Gy{\'a}rf{\'a}s~\cite{gyarfas1985chromatic} proved that circle graphs are $\chi$-bounded and asked~\cite{gyarfas1987problems,gyarfas1985covering} for improved $\chi$-bounding functions. In particular Gy{\'a}rf{\'a}s~\cite{gyarfas1987problems} originally asked if a linear $\chi$-bounding function was possible. This was answered in the negative by Kostochka~\cite{kostochka1988upper} who gave the superlinear lower bound of $\frac{1}{2}\omega (\ln \omega - 2)$. Matching this up to a constant factor, we prove the first $O(\omega \log \omega)$ upper bound.

\begin{theorem}\label{main}
	Every circle graph with clique number at most $\omega$ has chromatic number at most $2\omega \log_2 (\omega) +2\omega \log_2(\log_2 (\omega)) + 10\omega$.
\end{theorem}

Theorem~\ref{main} follows a number of improvements to the $\chi$-bounding function over the last 35 years~\cite{gyarfas1985chromatic,kostochka1988upper,kostochka1997covering,daviescircle}. In particular we extend and refine the techniques that were recently introduced by the author and McCarty~\cite{daviescircle} to prove the first polynomial $\chi$-bounding function of $O(\omega^2)$ for circle graphs.
Two other classes generalising circle graphs that are now known to have polynomially $\chi$-bounding functions are interval filament graphs~\cite{daviescircle,krawczyk2017line} and grounded $L$-graphs~\cite{davies2021colouring}. Combining Theorem~\ref{main} with a result in~\cite{krawczyk2017line} improves the best known $\chi$-bounding function for interval filaments graphs from $O(\omega^4)$ to $O(\omega^3\log \omega)$.
There are also many classes that contain circle graphs and are known to be $\chi$-bounded (although they usually have extremely large $\chi$-bounding functions), for the most general examples see~\cite{davies2020vertex,scott2020induced,rok2019coloring,chudnovsky2016induced}, and for more on $\chi$-boundedness see the recent survey by Scott and Seymour~\cite{scott2020survey}.

Circle graphs and their representations are fundamental objects that appear in a diverse range of study. Some examples include knot theory~\cite{bar1995vassiliev,birman1993knot},  bioinformatics~\cite{hofacker1998combinatorics}, quantum field theory~\cite{marie2013chord}, quantum computing~\cite{bravyi2007measurement,van2004graphical}, and data structures~\cite{flajolet1980sequence}. On the more combinatorial side, in addition to discrete and computational geometry, circle graphs and their representations also appear in the study of continued fractions~\cite{touchard1952probleme}, vertex-minors~\cite{geelen2020grid}, matroid representation~\cite{bouchet1987unimodularity}, and in various sorting problems~\cite{golumbic2004algorithmic}. Circle graphs are also deeply related to planar graphs; the fundamental graphs of planar graphs are exactly the class of bipartite circle graphs~\cite{de1984characterization}. Direct applications of colouring circle graphs include finding the minimum number of stacks needed to obtain a given permutation~\cite{even1971queues}, solving routing problems such as in VLSI physical design~\cite{sherwani2012algorithms}, and finding stack layouts of graphs, which also has a number of additional applications of its own (see~\cite{dujmovic2004linear}).

With these applications in mind, it is desirable to have an efficient algorithm for colouring circle graphs. While their clique numbers can be found in polynomial time~\cite{gavril1973}, unfortunately the problem of determining their chromatic number is NP-complete~\cite{garey1980complexity}. So the best that can be hoped for is an efficient approximation algorithm. The proof of Theorem~\ref{main} is constructive and yields a practical polynomial time algorithm for colouring circle graphs with a colouring that is optimal up to at most a logarithmic factor of the chromatic number.

For completeness, we also provide a new simple lower bound construction for the {$\chi$-bounding} function of circle graphs. As a bonus it improves Kostochka's~\cite{kostochka1988upper} lower bound by a factor of 2.

\begin{theorem}\label{lower}
	For every positive integer $\omega$ there is a circle graph with clique number at most $\omega$ and chromatic number at least $\omega(\ln \omega - 2)$.
\end{theorem}

For large clique number this leaves a constant factor of about $2\log_2(e) \approx 2.8854$ between the upper and lower bounds. These new upper and lower bounds are remarkably tight, but the difference between the logarithmic bases used in the upper and lower bounds is certainly curious.

For small clique number just one non-trivial tight bound is known; the maximum chromatic number of a triangle-free circle graph is equal to 5~\cite{kostochka1988upper,ageev1996triangle}.
In the next case, the best known upper bound is due to  Nenashev~\cite{nenashev2012upper} who proved that $K_4$-free circle graphs have chromatic number at most 30. 
By optimizing the proof of Theorem~\ref{main} to the $\omega=3$ case, it is possible to improve this upper bound to 19. We sketch the required modifications after the proof of Theorem~\ref{main}.

The proof of Theorem~\ref{main} is entirely self-contained and covered in the next few sections. In the last section we prove Theorem~\ref{lower}.

\section{Preliminaries}

The proof of Theorem~\ref{main} is essentially by proving a stronger statement on being able to extend certain well-structured partial pre-colourings. This better facilitates an inductive argument and is an idea most famously used in Thomassen's~\cite{thomassen1994every} proof that planar graphs are 5-chooseable. As in~\cite{daviescircle}, we use what we call a pillar assignment to colour our circle graphs. The reason for this is two-fold: pillar assignments provide a convenient way to describe the possible pre-colourings, and they also act as a useful tool for extending the pre-colourings.
However we require a definition of pillar assignments that is different to that of~\cite{daviescircle}.

In~\cite{daviescircle} we used pillar assignments to obtain an improper colouring such that every monochromatic component was a permutation graph.
By exploiting the structure of our improper colouring and using a natural Tur{\'a}n-type lemma on permutation graphs, we were able to bound the number of colours needed in this improper colouring.
Then finally a proper colouring was obtained by refining the improper colouring.

Although significantly easier to do so, obtaining a proper colouring by first going via this improper colouring appears to present a degree of inefficiency in minimising the number of colours used.
So the most significant difference with the notion of pillar assignment that we use is that it provides a proper colouring of the circle graph directly.
This involves colouring certain induced permutation subgraphs in a particular well-structured way.
The purpose of this additional structure in the colouring is to allow for a new Tur{\'a}n-type lemma. Although this lemma is less natural, it is much more specialised to our notion of pillar assignments.
With this new notion of pillar assignment and its tailor-made Tur{\'a}n-type lemma, we are then able to obtain the improved bounds with an inductive argument on extending pillar assignments in a similar way to in~\cite{daviescircle}.

As a step towards proving our required tailor-made Tur{\'a}n-type lemma, we actually prove a tight (and somewhat more abstract) version of the Tur{\'a}n-type lemma on permutation graphs used in~\cite{daviescircle} (see Theorem~\ref{estype}).

For convenience of proving Theorem~\ref{main} we use an interval overlap representation of our circle graphs rather than a chord diagram representation. An \emph{interval system} is a collection of open intervals in $(0,1)$ such that no two share an endpoint. Two distinct intervals $I_1,I_2$ \emph{overlap} if they have non-empty intersection, and neither is contained in the other. The \emph{overlap graph} of an interval system $\mathcal{I}$ is the graph with vertex set $\mathcal{I}$ where two vertices are adjacent whenever their corresponding intervals overlap. It can easily be checked that circle graphs are exactly overlap graphs of interval systems.
Similarly, permutation graphs are exactly the overlap graphs of interval systems $\mathcal{I}$ such that there exists a $p\in (0,1)$ with $p\in I$ for all intervals $I\in \mathcal{I}$.

It is often more convenient to examine properties of a circle graph as equivalent properties of their interval systems. Note that sets of pairwise non-overlapping intervals in an interval system correspond to stable sets in the overlap graph, and sets of pairwise overlapping intervals correspond to cliques.
Given an interval system $\mathcal{I}$, we let $\omega(\mathcal{I})$ be equal to the size of the largest set of pairwise overlapping intervals contained in $\mathcal{I}$.
Equivalently, $\omega(\mathcal{I})$ is equal to the clique number of the overlap graph of $\mathcal{I}$.
Similarly we consider colourings of an interval system with a notion equivalent to that of colourings of their overlap graphs. A proper partial colouring of an interval system $\mathcal{I}$ is an assignment of colours to a subset of the intervals of $\mathcal{I}$ so that no pair of overlapping intervals receive the same colour. We say that a proper colouring of $\mathcal{I}$ is \emph{complete} if every interval of $\mathcal{I}$ is assigned a colour.

For an interval $I\subseteq (0,1)$, let $\ell(I)$ be its leftmost endpoint and let $r(I)$ be its rightmost endpoint. For two intervals $I_1,I_2\subseteq (0,1)$, we use $I_1<I_2$ to denote that $r(I_1) < \ell(I_2)$, and similarly $I_1 > I_2$ to denote that $\ell(I_1) > r(I_2)$.
Given a finite partially ordered set $(X, \preceq)$, and some $x\in X$, the \emph{height} $h(x)$ of $x$ in the partial order is equal to the maximum length of a chain ending in $x$. For a positive integer $k$, we let $[k] = \{1,\dots , k\}$.

We finish this section with a lemma on colouring permutation graphs that is used in our definition of pillar assignments.
In addition to the bound on the number of colours required, we also make use of the described additional properties of this colouring.

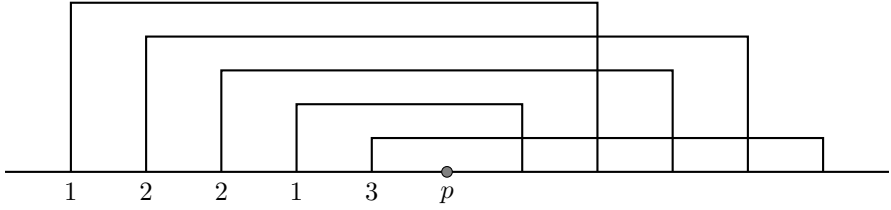
\begin{figure}
	\centering
	\begin{tikzpicture}
	\def \w {6} \def \h {.45} \def \s {.5}
	\tikzstyle{end} = [draw,circle, inner sep=.05cm]
	\tikzstyle{T} = [thick, dashed]
	\tikzstyle{One} = [thick]
	%draws labelled (0,1)
	\newcommand{\drawBase}{
		\node[] (A) at (-\w,0) {};
		\node[] (B) at (\w,0) {};
		\draw[thick] (A) -- (B);
	}
	%draws interval beginning at x-coord #1, ending #2 and height #3, style
	\newcommand{\drawI}[4]{
		\draw[#4] (#1,0) -- ++(0,#3) -- ++({#2+(-1)*(#1)},0) -- ++(0,-#3);
	}
	\drawI{-5}{2}{5*\h}{One}
	\node[label=below:{1}] at (-5,0.1) {};
	\drawI{-3}{3}{3*\h}{One}
	\node[label=below:{2}] at (-4,0.1) {};
	\drawI{-4}{4}{4*\h}{One}
	\node[label=below:{2}] at (-3,0.1) {};
	\drawI{-2}{1}{2*\h}{One}
	\node[label=below:{1}] at (-2,0.1) {};
	\drawI{-1}{5}{\h}{One}
	\node[label=below:{3}] at (-1,0.1) {};
	\drawBase
	\node[end, fill=gray, label=below:{$p$}] at (0,0) {};
	\end{tikzpicture}
	\caption{The colouring $\phi_p$ of an interval system whose intervals all contain $p$ in the case that $C=\{1,2,3\}$. The colour that the intervals receive is the number appearing below their leftmost endpoint.}
	\label{fig:pillA}
\end{figure}

\begin{lemma}\label{permutation colouring}
	Let $\mathcal{I}$ be an interval system with $\omega(\mathcal{I})= \omega$ such that all intervals of $\mathcal{I}$ contain some given point $p\in \mathbb{R}$, and let $C\subset \mathbb{N}$ have $|C|= \omega$.
	Then there is a proper colouring $\phi_p :\mathcal{I} \to C$ such that if $I_1,\dots, I_k \in \mathcal{I}$ are intervals with $\phi_p (I_1) < \dots < \phi_p (I_k)$ and $\ell(I_1) < \dots < \ell(I_k)$
	(or $r(I_1) < \dots < r(I_k)$),
	then there exist $k$ pairwise overlapping intervals $I_1^*, \dots I_k^*\in \mathcal{I}$ with $\ell(I_1^*), \dots , \ell(I_k^*) \in [\ell(I_1), \ell(I_k)]$ (or $r(I_1^*), \dots , r(I_k^*) \in [r(I_1), r(I_k)]$ respectively).
\end{lemma}

\begin{proof}
	Let $C=\{c_1, \dots , c_{\omega}\}$, where $c_1 < \dots < c_{\omega}$.
	Let $\preceq$ be the partial order of $\mathcal{I}$ such that $I \preceq I'$ whenever $\ell(I)< \ell(I')$ and $r(I) < r(I')$ (or $I=I'$). Notice that two intervals overlap exactly when they are comparable in the partially ordered set $(\mathcal{I}, \preceq)$.
	For each $I\in \mathcal{I}$, let $\phi_p (I)= c_{h(I)}$.
	The colouring $\phi_p$ is exactly the same as a first fit colouring of the intervals of $\mathcal{I}$ when they are ordered according to either their leftmost or rightmost endpoints.
	This provides a proper $C$-colouring as the size of the largest chain in $(\mathcal{I}, \preceq)$ is equal to $\omega$. For an example of such a colouring $\phi_p$, see Figure~\ref{fig:pillA}.
	It remains to show that this colouring satisfies the desired properties.
	
	Suppose that $I_1,\dots, I_k \in \mathcal{I}$ are such that $\phi_p(I_1) < \dots < \phi_p(I_k)$ and $\ell(I_1) < \dots < \ell(I_k)$.
	Let $I_k^*= I_k$ and  for each $j<k$ in decreasing order, let $I_j^*$ be the interval with $\ell(I_j)$ maximum, subject to $I_j^* \prec I_{j+1}^*$ and $\phi_p(I_j^*)=\phi_p(I_j)$. Such intervals $I_j^*$ must exist with $\ell(I_j) \le \ell(I_j^*) < \ell(I_{j+1}^*)$ by the choice of colouring $\phi_p$ as if $I'\in \mathcal{I}$ were an interval with $I' \prec I_{j+1}^*$, $\phi_p(I')=\phi_p(I_j)$, and $\ell(I')<  \ell(I_{j+1}^*)$, then $I'$ and $I_{j}$ would be non-overlapping, and so $I_{j}$ would overlap with $I_{j+1}^*$ and hence precede $I_{j+1}^*$ in the partial order. Then $I_1^* \prec \dots \prec I_k^*$, and so $I_1^*, \dots I_k^*$ are pairwise overlapping with $\ell(I_1^*), \dots , \ell(I_k^*) \in [\ell(I_1), \ell(I_k)]$ as required.
	
	The other case is very similar.
	Suppose that $I_1,\dots, I_k \in \mathcal{I}$ are such that $\phi_p(I_1) < \dots < \phi_p(I_k)$ and $r(I_1) < \dots < r(I_k)$.
	As before, let $I_k^*= I_k$ and for each $j<k$ in decreasing order, let $I_j^*$ be the interval with $r(I_j)$ maximum, subject to $I_j^* \prec I_{j+1}^*$ and $\phi_p(I_j^*)=\phi_p(I_j)$.
	As before, such intervals $I_j^*$ must exist with $r(I_j) \le r(I_j^*) < r(I_{j+1}^*)$ by the choice of colouring $\phi_p$. So $I_1^* \prec \dots \prec I_k^*$ are again pairwise overlapping with $r(I_1^*), \dots , r(I_k^*) \in [r(I_1), r(I_k)]$ as required.
\end{proof}

\section{Pillar assignments}

We start this section by defining our notion of pillar assignments, the tool we use to colour circle graphs. The colouring in Lemma~\ref{permutation colouring} is crucial to the definition of pillar assignments and thus crucial for colouring our circle graphs. Afterwards we examine some properties of pillar assignments.

A \emph{pillar} of an interval system $\mathcal{I}$ is a point within $(0,1)$ that is distinct from the endpoints of the intervals of $\mathcal{I}$.
For totally ordered pillars $(P,\preceq)$, we say that an interval $I\in \mathcal{I}$ is \emph{assigned} to a pillar $p\in P$ if $p\in I$ and there is no pillar $p'\in P$ such that $p'\in I$ and $p' \prec p$. So every interval is assigned to at most one pillar.
For each pillar $p\in P$, we let $\mathcal{I}_p$ be the intervals of $\mathcal{I}$ that are assigned to $p$.
The \emph{foundation} $F_p$ of a pillar $p\in P$ is the open interval containing $p$ that has its endpoints in $\{p'\in P : p' \prec p\}\cup \{0,1\}$ and contains no pillar $p'\in P$ with $p' \prec p$.

Next we show how to obtain a proper partial colouring of an interval system $\mathcal{I}$ from a collection of totally ordered pillars $(P,\preceq)$.
We refer the reader to Figure~\ref{fig:pillB} for an illustration of a pillar assignment and the colouring obtained from ordered pillars.
For each pillar $p \in P$ in order, we assign a set of colours $C_p\subset \mathbb{N}$ to $p$ and a $C_p$-colouring $\phi_p: \mathcal{I}_p \to C_p$ of the intervals assigned to $p$ as follows.

If $p$ is the first pillar in the total order $\preceq$, then let $C_p= \left\{1,\dots, \omega(\mathcal{I}_p) \right\}$, and let $\psi_p = \phi_p$ be a $C_p$-colouring of $\mathcal{I}_p$ as in Lemma~\ref{permutation colouring}.

Otherwise let $p^*$ be the pillar immediately preceding $p$ in the total order $\preceq$.
Then let $\mathcal{F}_p$ be the intervals of $\mathcal{I}$ that have exactly one endpoint in $F_p$. Let $C_p$ be the set of the smallest $\omega(\mathcal{I}_p)$ positive integers that are not contained in $\psi_{p^*}(\mathcal{F}_p)$.
Then let $\phi_p$ be a $C_p$-colouring of $\mathcal{I}_p$ as in Lemma~\ref{permutation colouring}. Let $\psi_p=\psi_{p^*}\cup \phi_p$. Note that $\psi_p$ remains a proper partial colouring of $\mathcal{I}$ as the intervals of $\mathcal{I}_p$ are all contained in the foundation $F_p$, and so do not overlap with any of the intervals $\psi_{p^*}^{-1}(C_p)$ by the choice of $C_p$.

Then for the last pillar $q$ in the total order $\preceq$, we let $\psi_{(P,\preceq)}=\psi_q$. Another convenient equivalent definition which we often use is $\psi_{(P,\preceq)} = \bigcup_{p\in P} \phi_p$.

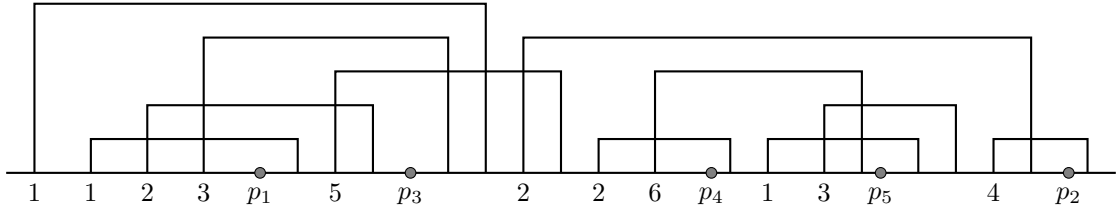
\begin{figure}
	\centering
	\begin{tikzpicture}
	\def \w {6} \def \h {.45} \def \s {.5}
	\tikzstyle{end} = [draw,circle, inner sep=.05cm]
	\tikzstyle{T} = [thick, dashed]
	\tikzstyle{One} = [thick]
	%draws labelled (0,1)
	\newcommand{\drawBase}{
		\node[] (A) at (-7.5,0) {};
		\node[] (B) at (7.5,0) {};
		\draw[thick] (A) -- (B);
	}
	%draws interval beginning at x-coord #1, ending #2 and height #3, style
	\newcommand{\drawI}[4]{
		\draw[#4] (#1,0) -- ++(0,#3) -- ++({#2+(-1)*(#1)},0) -- ++(0,-#3);
	}
	\drawBase
	\node[end, fill=gray, label=below:{$p_1$}] at (-4,0) {};
	\drawI{-7}{-1}{5*\h}{One}
	\node[label=below:{1}] at (-7,0.1) {};
	\drawI{-6.25}{-3.5}{\h}{One}
	\node[label=below:{1}] at (-6.25,0.1) {};
	\drawI{-5.5}{-2.5}{2*\h}{One}
	\node[label=below:{2}] at (-5.5,0.1) {};
	\drawI{-4.75}{-1.5}{4*\h}{One}
	\node[label=below:{3}] at (-4.75,0.1) {};
	\node[end, fill=gray, label=below:{$p_3$}] at (-2,0) {};
	\drawI{-3}{0}{3*\h}{One}
	\node[label=below:{5}] at (-3,0.1) {};
	\node[end, fill=gray, label=below:{$p_4$}] at (2,0) {};
	\drawI{0.5}{2.25}{\h}{One}
	\node[label=below:{2}] at (0.5,0.1) {};
	\drawI{-0.5}{6.25}{4*\h}{One}
	\node[label=below:{2}] at (-0.5,0.1) {};
	\drawI{1.25}{4}{3*\h}{One}
	\node[label=below:{6}] at (1.25,0.1) {};
	\node[end, fill=gray, label=below:{$p_2$}] at (6.75,0) {};
	\drawI{5.75}{7}{\h}{One}
	\node[label=below:{4}] at (5.75,0.1) {};
	\node[end, fill=gray, label=below:{$p_5$}] at (4.25,0) {};
	\drawI{2.75}{4.75}{\h}{One}
	\node[label=below:{1}] at (2.75,0.1) {};
	\drawI{3.5}{5.25}{2*\h}{One}
	\node[label=below:{3}] at (3.5,0.1) {};
	\end{tikzpicture}
	\caption{The colouring $\psi_{(P,\preceq )}$ of an interval system $\mathcal{I}$ for a collection of totally ordered pillars $(P,\preceq )$ with $P=\{p_1,p_2,p_3,p_4,p_5\}$ and $p_1 \prec p_2 \prec p_3 \prec p_4 \prec p_5$. The colour that the intervals receive is the number appearing below their leftmost endpoint. In this example we get that $C_{p_1}=\{1,2,3\}$, $C_{p_2}=\{4\}$, $C_{p_3}=\{5\}$, $C_{p_4}=\{2,6\}$, and $C_{p_5}=\{1,3\}$.}
	\label{fig:pillB}
\end{figure}

A \emph{pillar assignment} of an interval system $\mathcal{I}$ is a triple $(P,\preceq, \psi)$ such that $P$ is a set of pillars, $\preceq$ is a total ordering of $P$, and $\psi$ is the proper partial colouring $\psi_{(P,\preceq)}$ of $\mathcal{I}$ as described above.
A pillar assignment $(P,\preceq, \psi)$ is \emph{complete} if every interval of $\mathcal{I}$ contains some pillar of $P$ (or equivalently if $\psi$ colours every interval of $\mathcal{I}$).
For a pillar assignment $(P,\preceq, \psi)$, let $\chi(P,\preceq, \psi)$ be equal to $\left| \bigcup_{p\in P} C_p \right| =|\psi(\mathcal{I})|$, in other words $\chi(P,\preceq, \psi)$ is the number of colours that the pillar assignment $(P,\preceq, \psi)$ uses to colour its interval system $\mathcal{I}$. So if $(P,\preceq, \psi)$ is a complete pillar assignment, then $\chi(\mathcal{I}) \le \chi(P,\preceq, \psi)$. By the above definition and discussion, we have the following.

\begin{lemma}\label{pillarcolouring}
	Let $(P,\preceq,\psi)$ be a pillar assignment of an interval system $\mathcal{I}$. Then $\psi$ is a proper partial colouring of $\mathcal{I}$ that colours every interval containing a pillar of $P$, and furthermore if the pillar assignment is complete then $\psi$ is a complete proper colouring of the interval system $\mathcal{I}$, and so $\chi(\mathcal{I}) \le \chi(P,\preceq, \psi)$.
\end{lemma}

Next we analyse the endpoints of chords assigned to a given pillar, and also the endpoints of chords with a given colour in a pillar assignment.
These two lemmas are used in the proof of our tailor-made Tur{\'a}n-type result in Section 4.
An \emph{arch} of a pillar assignment $(P,\preceq, \psi)$ is an open interval with endpoints in $\{0,1\}\cup P$ that contains no pillar of $P$.

\begin{lemma}\label{chordsscomefrom}
	Let $(P,\preceq, \psi)$ be a pillar assignment of an interval system $\mathcal{I}$, let $K$ be an arch of $(P,\preceq, \psi)$, and let $\mathcal{I}_K$ be the intervals of $\mathcal{I}$ with exactly one endpoint in $K$. Let $\bigcup _{p\in P} \mathcal{I}_{(K,p)}$ be the partition of $\mathcal{I}_K$ where for each $p\in P$, the intervals $\mathcal{I}_{(K,p)}$ are exactly the intervals of $\mathcal{I}_K$ that are assigned to pillar $p$. Then there is a collection of disjoint intervals $\{K_p : p\in P\}$ contained in $(0,1)\backslash K$ such that for every $p\in P$, the intervals of $\mathcal{I}_{(K,p)}$ have an endpoint within $K_p$.
\end{lemma}

\begin{proof}
	Let $K=(k^-,k^+)$. First observe that for each pillar $p\in P$, the intervals of $\mathcal{I}_{(K,p)}$ must all be contained in either $(k^-,1)$ or $(0,k^+)$ depending on if the pillar $p$ is contained in $[k^+,1)$ or $(0,k^-]$.
	
	Now suppose for sake of contradiction that no such collection of disjoint intervals $ \{K_p : p\in P\}$ exist. Then there must exist two distinct pillars $p,p'$ and distinct intervals $I_1, I_2 \in \mathcal{I}_{(K,p)}$, $I'\in \mathcal{I}_{(K,p')}$ such that the endpoints $e_1,e_2.e'$ of $I_1,I_2,I'$ respectively that are contained in $(0,1)\backslash K$ are such that $e_1 < e' < e_2$, and either $e_1 < e' < e_2 < k^-$, or $k^+ < e_1 < e' < e_2$.
	
	Suppose in the first case that $e_1 < e' < e_2 < k^-$. Then $I'\backslash K$ must contain $p'$, and furthermore both $I_1\backslash K$ and $I_2 \backslash K$ must contain $p$. Hence $I_1\backslash K$ contains both $p$ and $p'$. As $I_1$ is assigned to $p$, we see that $p\prec p'$. Then $I'$ does not contain $p$ as $I'$ is assigned to $p'$ and $p\prec p'$. But this contradicts the fact that $I_2\backslash K \subset I'\backslash K$ contains $p$.
	The second case that $k^+ < e_1 < e' < e_2$ is argued similarly and we conclude that such a collection of disjoint intervals $ \{K_p : p\in P\}$ exists.
\end{proof}

\begin{lemma}\label{samecolourchords}
	Let $(P,\preceq, \psi)$ be a pillar assignment of an interval system $\mathcal{I}$, let $K$ be an arch of $(P,\preceq, \psi)$, and let $\mathcal{I}_K$ be the intervals of $\mathcal{I}$ with exactly one endpoint in $K$.
	Let $\bigcup _{p\in P} \mathcal{I}_{(K,p)}$ be the partition of $\mathcal{I}_K$ where for each $p\in P$, the intervals $\mathcal{I}_{(K,p)}$ are exactly the intervals of $\mathcal{I}_K$ that are assigned to pillar $p$.
	For each $c\in \psi(\mathcal{I})$, let $\mathcal{I}_{(K,c)}$ be the intervals of $\mathcal{I}_K$ that are coloured $c$ by $\psi$.
	Then for each $c\in \psi(\mathcal{I})$, there is a pillar $p\in P$, such that $\mathcal{I}_{(K,c)} \subseteq \mathcal{I}_{(K,p)}$.
\end{lemma}

\begin{proof}
	Suppose not, then there must exist intervals $I_1,I_2\in \mathcal{I}_K$ such that $\psi(I_1) = \psi(I_2)$ and $I_1,I_2$ are assigned to distinct pillars $p_1,p_2$. In particular this means that $\phi_{p_1}(I_1)= \phi_{p_2}(I_2)$ as in the definition of the colouring $\psi=\psi_{(P,\preceq)}$.
	
	Without loss of generality, we may assume that $p_1 \prec p_2$.
	Then the foundation $F_{p_2}$ of $p_2$ must contain $I_2$, and so $F_{p_2}$ contains $K$ as well.
	Hence the foundation $F_{p_2}$ contains an endpoint of $I_1$.
	But this now contradicts the choice of $C_{p_2}$ by the definition of the colouring $\psi= \psi_{(P,\preceq)}$.
\end{proof}

Next we define a notion for the degree of an interval $J$ contained within an arch of a pillar assignment $(P,\preceq, \psi)$ of some interval system $\mathcal{I}$. It is this notion of degree that our tailor-made Tur{\'a}n-type result is based on.

For an interval $J$ within an arch of a pillar assignment $(P,\preceq, \psi)$ of an interval system $\mathcal{I}$, the \emph{degree} $d_{(P,\preceq, \psi)}(J)$ of $J$ is equal to the number of colours that intervals of $\mathcal{I}$ with an endpoint in $J$ receive from $\psi$.
As an example, for the pillar assignment $(P,\preceq, \psi)$ depicted in Figure~\ref{fig:pillB}, $d_{(P,\preceq, \psi)}(p_3,p_4)=5$, and $d_{(P,\preceq, \psi)}((p_3,p_4), \{p_1,p_2,p_3\})= 2 + 0 + 1 = 3$.
When the pillar assignment is clear from context we often omit the subscript on the chromatic and clique degrees.

A pillar assignment $(P^*,\preceq^*, \psi^*)$ \emph{extends} a pillar assignment ${(P,\preceq, \psi)}$ if $P\subset P^*$, every pillar of $P$ precedes every pillar of $P^*\backslash P$ in $\preceq^*$ and ${(P^*, \preceq^*)|_P = (P,\preceq)}$, and $\psi^*$ is a proper partial colouring that extends $\psi$. We remark that by definition, the last condition that $\psi^*$ extends $\psi$ is implied by the conditions on $(P^*,\preceq^*)$, because for the interval system $\mathcal{I}$, the colourings $\psi=\psi_{(P,\preceq )}$ and $\psi^*=\psi_{(P^*,\preceq^*)}$ are determined solely by the totally ordered pillars $(P,\preceq )$ and $(P^*,\preceq^*)$ respectively.

We finish this section with a divide and conquer lemma that under favourable conditions allows for a certain extension of a pillar assignment that maintains a low total number of colours used and low degree arches.

\begin{lemma}\label{devide and conquer}
	Let $(P,\preceq, \psi)$ be a pillar assignment of an interval system $\mathcal{I}$ with $\omega(\mathcal{I})=\omega$, let $K$ be an arch, let $t$ be a positive integer, and let $Q \subset K$ be a finite collection of pillars such that $d_{(P,\preceq, \psi)}(J) \le t$ for every interval $J$ contained in $K\backslash Q$.
	Then there is a pillar assignment $(P^*,\preceq^*, \psi^*)$ extending $(P,\preceq, \psi)$ such that:
	\begin{itemize}
		\item $P^*=P\cup Q$,
		
		\item $d_{(P^*,\preceq^*, \psi^*)}(J) \le t+ \omega  \lceil \log_2(|Q|+1) \rceil$ for every interval $J \subset K\backslash Q$, and
		
		\item $\chi(P^*,\preceq^*, \psi^*) \le \max\{\chi(P,\preceq, \psi), \  d_{(P,\preceq, \psi)}(K) + \omega  \lceil \log_2(|Q| + 1) \rceil\}$.
	\end{itemize}
\end{lemma}

\begin{proof}
	Firstly the result is trivially true if $|Q|= 0$. So from here we argue inductively on $|Q|$.
	
	Let the endpoints of $K$ be $q_0$ and $q_n$, with $q_0<q_n$ and $n=|Q|+1$. Then let the elements of $Q$ be $\{q_1,\dots, q_{n-1}\}$ where $q_1 < \dots <q_{n-1}$.
	Consider the pillar assignment $(P',\preceq', \psi')$ extending $(P,\preceq, \psi)$ that is obtained by adding the pillar $q_{\lceil \frac{n-1}{2}  \rceil}$ immediately after all the pillars of $P$ in the total ordering $(P,\preceq)$.
	
	Then with respect to the pillar assignment $(P',\preceq', \psi')$, the interval $K$ contains exactly two arches; $K_1=(q_0, q_{\lceil \frac{n-1}{2}  \rceil})$, and $K_2=(q_{\lceil \frac{n-1}{2}  \rceil}, q_n)$.
	By considering the colouring $\phi_{q_{\lceil \frac{n-1}{2}  \rceil}} = \psi_{(P',\preceq')} \backslash \psi_{(P,\preceq)}$, we can observe that:
	\begin{itemize}
		\item $d_{(P',\preceq', \psi')}(K_1), \  d_{(P',\preceq', \psi')}(K_2) \le d_{(P,\preceq, \psi)}(K) + \omega$,
		
		\item $\chi(P',\preceq', \psi') \le {\max\{\chi(P,\preceq, \psi), \  d_{(P,\preceq, \psi)}(K) + \omega  \}}$, and
		
		\item ${d_{(P',\preceq', \psi')}(q_i,q_{i-1}) \le t + \omega}$ for every $i\in [n]$.
	\end{itemize}
	
	Next note that respect to the colouring, extending the pillar assignment $(P',\preceq', \psi')$ within each of the arches $K_1$ and $K_2$ is independent of the other.
	So we may apply the result of the inductive hypothesis twice, once to the pillars $\{q_1,\dots, q_{\lceil \frac{n-1}{2}  \rceil -1}\}\subset K_1$, and then to the pillars $\{q_{\lceil \frac{n-1}{2}  \rceil +1},\dots q_{n-1}\}\subset K_2$, to obtain a new pillar assignment ${(P^*,\preceq^*, \psi^*)}$ extending  $(P',\preceq', \psi')$ (and so also extending $(P,\preceq, \psi)$).
	Furthermore the resulting pillar assignment $(P^*,\preceq^*, \psi^*)$ is such that: $P^*=P'\cup \{q_1,\dots, q_{\lceil \frac{n-1}{2}  \rceil -1}\} \cup \{q_{\lceil \frac{n-1}{2}  \rceil +1},\dots q_{n-1}\} = P \cup Q$, and
	for each $i\in [n]$,
	\begin{align*}
	d_{(P^*,\preceq^*, \psi^*)}(q_{i-1}, q_i)
	&\le t+ \omega + \omega \left\lceil \log_2\left( \max \left\{ \left\lceil \frac{n-1}{2}  \right\rceil  , \  n - \left\lceil \frac{n-1}{2}  \right\rceil \right\} \right) \right\rceil \\
	&\le t+ \omega \lceil \log_2(n) \rceil \\
	&= t+ \omega  \lceil \log_2(|Q|+1) \rceil,
	\end{align*}
	and lastly,
	\begin{align*}
	\chi(P^*,\preceq^*, \psi^*)
	&\le \max\left\{
	\begin{aligned}
	& \qquad \qquad \qquad \  \chi(P',\preceq', \psi'), \\
	& \ \  d_{(P',\preceq', \psi')}(K_1) + \omega  \left\lceil \log_2\left( \left\lceil \frac{n-1}{2}  \right\rceil \right) \right\rceil, \\
	&d_{(P',\preceq', \psi')}(K_2) + \omega  \left\lceil \log_2 \left( n - \left\lceil \frac{n-1}{2}  \right\rceil \right) \right\rceil
	\end{aligned}
	\right\} \\
	&\le \max\left\{\chi(P,\preceq, \psi), \  d_{(P,\preceq, \psi)}(K) +  \omega \lceil \log_2(n) \rceil \right\} \\
	&= \max\left\{\chi(P,\preceq, \psi), \ d_{(P,\preceq, \psi)}(K) +  \omega \lceil \log_2(|Q|+1) \rceil \right\}.
	\end{align*}
	Hence $(P^*,\preceq^*, \psi^*)$ provides the desired pillar assignment.
\end{proof}

\section{Extremal results}

In this section we prove the Tur{\'a}n-type lemma that is tailor-made for our notion of pillar assignment and degree (Lemma~\ref{extremal}). The purpose of this is to enable prudent usage of Lemma~\ref{devide and conquer} in the proof of our main result, a strengthening of Theorem~\ref{main} that concerns extending pillar assignments.
The idea of Lemma~\ref{extremal} is based on a Tur{\'a}n-type theorem of Capoyleas and Pach~\cite{capoyleas1992turan} for circle graphs. For an interval system $\mathcal{I}$ and a collection of disjoint intervals $\mathcal{J}$, the Tur{\'a}n-type theorem of Capoyleas and Pach~\cite{capoyleas1992turan} bounds (in terms of $\omega(\mathcal{I})$ and $|\mathcal{J}|$) the number of pairs of distinct intervals $J_1,J_2\in \mathcal{J}$ such that there is an interval of $\mathcal{I}$ with an endpoint in both $J_1$ and $J_2$.

First we need to prove Theorem~\ref{estype}, a theorem in a similar style to that of the Erd\H{o}s-Szekeres theorem~\cite{erdos1935combinatorial}. This result may be of independent interest. Indeed it can be shown that Theorem~\ref{estype} is equivalent to a tight version of the Tur{\'a}n-type lemma used in~\cite{daviescircle}, so it can also be considered an exact permutation graph analogue of the aforementioned Tur{\'a}n-type theorem of Capoyleas and Pach~\cite{capoyleas1992turan}. With a bit more care one can even characterise the extremal examples.

Before stating and proving Theorem~\ref{estype}, we first require two definitions and a simple lemma.
Given some $S\subseteq \mathbb{R}^d$, we define the \emph{strong dominance partial ordering} $\preceq_{sd}$ of $S$ to be the partial order such that $u\preceq_{sd} v$ exactly when each coordinate of $v$ is greater than the corresponding coordinate of $u$ (or $u=v$). Given two sets $A$ and $B$, we let $A \times B$ denote the \emph{Cartesian product} $\{(a,b) : a\in A \text{ and } b\in B\}$, of $A$ and $B$.

\begin{lemma}\label{gridantichain}
	Let $a,b$ be positive integers, and let $\preceq_{sd}$ be the strong dominance partial ordering of $[a]\times [b]$.
	Then the maximum length of an antichain in $([a]\times [b],\preceq_{sd})$ is equal to $a+b-1$.
\end{lemma}

\begin{proof}
	Let $A$ be the antichain $\{(a,j) : j\in [b]\} \cup \{(i,b) : i\in [a-1]\}$, then $|A|=a+b-1$. For each integer $k$ with $1-b \le k \le a -1$, let $C_k=\{(x,y)\in  [a]\times [b] : x-y =k \}$. Then $C_{1-b}, \dots , C_{a-1}$ is a chain cover of $[a]\times [b]$ of size $a+b-1$. An antichain contains at most one elements of every chain in a chain cover.
	Hence the maximum length of an antichain in $([a]\times [b],\preceq_{sd})$ is equal to $|A|=a+b-1$ as required.
\end{proof}

\begin{theorem}\label{estype}
	Let $a, b, n$ be positive integers with $n\le a, b$, and let $\prec_{sd}$ be the strong dominance partial ordering of $[a]\times [b]$.
	Let $S\subseteq [a]\times [b]$ be a set containing no chain of length greater than $n$.
	Then $|S| \le  n(a+b -n)$.
\end{theorem}

\begin{proof}
	Let $m$ be the maximum length of a chain contained in $(S,\preceq_{sd})$, and let $A_1,\dots, A_m$ be the antichains cover of $S$ where $A_k=\{(x,y)\in S : h_{(S,\preceq_{sd})}(x,y) =k  \}$ for each $k\in [m]$.
	Then for each $k\in [m]$, and each $(x,y)\in A_k$, there exists a chain $C_{(x,y)}$ of length $k$ ending in $(x,y)$. This implies that $x,y \ge k$. So for each $k\in [m]$, the antichain $A_k$ is contained in the grid $([a]\backslash[k-1]) \times ([b]\backslash [k-1])$. Then by Lemma~\ref{gridantichain}, $|A_k|\le (a-k+1) + (b -k + 1) -1= a+b -2k +1$ for every $k\in [m]$.
	
	Lastly
	\[
	|S|
	=  \sum_{k=1}^{m} |A_k| 
	\le \sum_{k=1}^{m} \left( a+b -2k + 1 \right)
	= m(a+b -m)
	\le n(a+b -n)
	\]
	as desired.
\end{proof}

The bound in this theorem is tight: one extremal example is $\{(x,y)\in [a]\times [b] : x \le n \text{ or } y \le n \}$, which contains no chain of length greater than $n$. The theorem can also be generalised to higher dimensional grids with essentially the same proof. We anticipate that Theorem~\ref{estype} will likely also find further applications in proving improved $\chi$-bounding functions for other classes of geometric intersection graphs. 

We now proceed with applying Theorem~\ref{estype} to prove our tailor-made Tur{\'a}n-type lemma.

\begin{lemma}\label{extremal}
	Let $(P,\preceq, \psi)$ be a pillar assignment of an interval system $\mathcal{I}$ with $\omega(\mathcal{I})=\omega$, let $K$ be an arch such that $d(K)\ge \omega$, and
	let $\mathcal{J}$ be a collection of disjoint open intervals contained within $K$ such that $|\mathcal{J}| \ge  \omega$.
	Then
	\[
	\sum_{J\in \mathcal{J}}d(J) \le   \omega(d(K) +|\mathcal{J}| - \omega   ).
	\]
\end{lemma}

\begin{proof}
	Let $\mathcal{I}_K$ be the intervals of $\mathcal{I}$ with exactly one endpoint in $K$. Let $\bigcup _{p\in P} \mathcal{I}_{(K,p)}$ be the partition of $\mathcal{I}_K$ where for each $p\in P$, the intervals $\mathcal{I}_{(K,p)}$ are exactly the intervals of $\mathcal{I}_K$ that are assigned to pillar $p$.
	Let $P'$ be the set of pillars $p \in P$ such that $\mathcal{I}_{(K,p)}$ is non-empty.
	Then by Lemma \ref{chordsscomefrom} there is a collection of disjoint intervals $\{K_p : p\in P'\}$ contained in $(0,1)\backslash K$ such that for every $p\in P'$, the intervals of $\mathcal{I}_{(K,p)}$ have an endpoint within $K_p$.
	Let $\preceq_{K}$ be the total ordering of $\{K_p : p\in P'\}$ so that $K_{p_1} \prec_K K_{p_2}$ exactly when either $K_{p_1} < K_{p_2} < K$, or $K < K_{p_1} < K_{p_2}$, or $K_{p_2}  < K  < K_{p_1}$.
	The key property of this total order is that if $p,p'$ are distinct pillars of $P'$ with $p\prec p'$, and $I,I'\in \mathcal{I}_K$ are intervals assigned to $p$ and $p'$ respectively, then $I$ overlaps with $I'$ if the endpoint of $I$ in that is contained in $K$ precedes precedes the endpoint of $I_2$ that is contained in $K$.

	Let $C'= \psi(\mathcal{I}_K)$.
	For each $c\in C'$, let $\mathcal{I}_{(K,c)}$ be the intervals of $\mathcal{I}_K$ that are coloured $c$ by $\psi$.
	By Lemma \ref{samecolourchords}, for each $c\in C'$, there exists a pillar $p\in P'$ such that $\mathcal{I}_{(K,c)} \subseteq \mathcal{I}_{(K,p)}$, and in particular every interval of $\mathcal{I}_{(K,c)}$ has an endpoint in $K_p$.
	Now let $\preceq_C$ be the total ordering of $C'$ such that for every $c_1,c_2\in C'$, we have that $c_1 \preceq c_2$ exactly when either there exists a pillar $p\in P'$ such that $\mathcal{I}_{(K,c_1)}, \mathcal{I}_{(K,c_2)} \subseteq \mathcal{I}_{(K,p)}$, and $c_1 \le c_2$, or there exists distinct pillars $p_1,p_2\in P'$ such that $\mathcal{I}_{(K,c_1)} \subseteq \mathcal{I}_{(K,p_1)}$,   $\mathcal{I}_{(K,c_2)} \subseteq \mathcal{I}_{(K,p_2)}$ and $p_1 \prec_K p_2$.
	Let $f:C' \to [d(K)]$ be the bijection such that $f(c_1)\le f(c_2)$ exactly when $c_1 \preceq_C c_2$.
	
	Now let $\mathcal{J} = \{J_1, \dots , J_{|\mathcal{J}|}\}$ where $J_1 < \cdots <  J_{|\mathcal{J}|} $.
	Next let $S$ be the set of all elements $(x,y)\in [d(K)] \times [|\mathcal{J}|]$ such that there is an interval $I$ of $\mathcal{I}$ that is coloured $f^{-1}(x)$ by $\psi$, and has an endpoint in $J_y$.
	Note that $|S| = \sum_{J\in \mathcal{J}}d(J)$.
	
	Suppose now for sake of contradiction that
	\[
	\sum_{J\in \mathcal{J}}d(J) >   \omega(d(K) +|\mathcal{J}| - \omega   ).
	\]
	Then by Theorem~\ref{estype}, there is a chain $W$ contained in $S$ of length at least $\omega + 1$.
	Since this chain is contained in $S$, there must exist colours $c_1 \prec_C \cdots \prec_C c_{\omega + 1}$ contained in $C'$, and integers $1\le x_1 < \cdots < x_{\omega + 1} \le |\mathcal{J}|$ so that for each $j\in [\omega + 1]$, there is an interval $I_j \in \mathcal{I}$ with an endpoint $e_j$ contained in $J_{x_j}$ and $\psi(I_j)=c_j$.
	Since $J_{x_1} < \cdots < J_{x_{\omega +1   }  }$, we have that $e_1 < \cdots < e_{\omega +1   } $.
	
	Let $p_1,\dots ,p_n\in P'$ be the collection of distinct pillars so that for some integers $0=a_0 < \dots < a_{n-1} < a_n = \omega +1 $, we have that for each $p_i$, the intervals $\mathcal{I}_{p_i}=\{I_j : a_{i-1} < j \le a_i\}$ are all assigned to pillar $p_i$.
	Note that $p_1 \prec_{K} \cdots \prec_{K} p_n$ by the definition of the total ordering $(C', \preceq_C)$.
	
	Then by Lemma \ref{permutation colouring} and the definition of the pillar assignment $(P,\preceq, \psi)$, for each $i\in [n]$, there exist pairwise overlapping intervals $I_{a_{i-1}+1}^*, \dots , I_{a_i}^*$ that are all assigned to the pillar $p_i$, and all have an endpoint contained in $[e_{a_{i-1}+1} , e_{a_i}]$.
	Since $p_1 \prec_{K} \dots \prec_{K} p_n$, and $[e_1, e_{a_1}] < \dots < [e_{a_{n-1}+1}, e_{\omega +1 }]$, we see that any two intervals of $\{I_1^*,\dots, I_{\omega +1}^*\}$ that are assigned to distinct pillars also overlap.
	Hence the intervals $I_1^*,\dots , I_{\omega +1}^*$ pairwise overlap, a contradiction to the fact that $\omega(\mathcal{I})=\omega$.
\end{proof}

\section{Main result}

By Lemma~\ref{pillarcolouring}, the following theorem strengthens and so implies Theorem~\ref{main}, that every circle graph with clique number at most $\omega$ has chromatic number at most $2\omega \log_2 (\omega) +2\omega \log_2(\log_2 (\omega)) + 10\omega$.

\begin{theorem}\label{induction}
	Let $\omega\ge 2$ be a positive integer,
	let $G$ be a circle graph with clique number at most $\omega$, and let $\mathcal{I}$ be an interval system with overlap graph $G$.
	Let $(P,\preceq, \psi)$ be a pillar assignment of $\mathcal{I}$ such that $\chi(P,\preceq, \psi) \le 2\omega \log_2 (\omega) +2\omega \log_2(\log_2 (\omega)) + 10\omega$, and $d(K)_{(P,\preceq, \psi)}\le \omega \log_2 (\omega) + \omega \log_2(\log_2 (\omega)) + 6\omega$ for every arch $K$ of $(P,\preceq, \psi)$.
	Then there is a complete pillar assignment $(P^*,\preceq^*, \psi^*)$ of $\mathcal{I}$ extending $(P,\preceq, \psi)$ with ${\chi(P^*,\preceq^*, \psi^*)} \le 2\omega\log_2 (\omega) +2\omega \log_2(\log_2 (\omega)) + 10\omega$.
\end{theorem}

\begin{proof}
	The theorem is trivially true if $(P,\preceq, \psi)$ is a complete pillar assignment, so we proceed by induction on the number of intervals that are not coloured by $\psi=\psi_{(P,\preceq)}$.
	
	Let $K$ be an arch of $(P,\preceq, \psi)$ that contains some interval $I$ of $\mathcal{I}$. Then $I$ is not coloured by $\psi$. Let $q^*$ be a pillar contained in $I$.
	Let $q_0=\ell (K)$.
	Now for each integer $i\ge 1$ in increasing order, if the pillar $q_{i-1}$ was chosen and $d_{(P,\preceq, \psi)}(q_{i-1}, r(K)) > 2\omega$, then we choose the next pillar $q_i \in K$ so that $q_{i} > q_{i-1}$ and $d_{(P,\preceq, \psi)}(q_{i-1}, q_i) = 2\omega$. Note that such a $q_i$ can always be chosen if $d_{(P,\preceq, \psi)}(q_{i-1}, r(K)) > 2\omega$ as incrementally increasing some $q>q_{i-1}$ increases the degree of $(q_{i-1},q)$ by at most 1. Let $n$ be equal to the largest $i$ such that the pillar $q_i$ is chosen, and let $Q=\{q_1, \dots, q_n, q^*\}$.
	Then $d_{(P,\preceq, \psi)}(q_{n-1}, r(K)) > 2\omega$.
	
	Let $\mathcal{J}=\{(q_0,q_1),\dots, (q_{n-1},r(K))\}$.
	Then $\sum_{J\in \mathcal{J}}d(J) > 2\omega n$. So by Lemma~\ref{extremal},
	\[
	2\omega n < \sum_{J\in \mathcal{J}}d(J) \le \omega (d(K)+n-\omega)\le \omega (\omega \log_2 (\omega) + \omega \log_2(\log_2 (\omega)) + 5\omega + n).
	\]
	Hence $n < \omega \log_2 (k) + \omega \log_2(\log_2 (\omega)) + 5\omega$, and so $|Q|< \log_2 (\omega) + \omega \log_2(\log_2 (\omega)) + 5\omega +1$.

	Then by Lemma~\ref{devide and conquer} there is a pillar assignment $(P',\preceq', \psi')$ extending $(P,\preceq, \psi)$ such that $P'= P\cup Q$, and for every arch $K'$ of $(P',\preceq', \psi')$ contained in $K$,
	\begin{align*}
	d_{(P',\preceq', \psi')}(K')
	&\le 2\omega + \omega  \lceil \log_2 (|Q|+1) \rceil \\
	&< 3\omega + \omega \log_2(\omega \log_2 (\omega) +\omega \log_2(\log_2 (\omega)) + 5\omega +2 ) \\
	&\le 3\omega + \omega \log_2(8\omega \log_2 (\omega)) \\
	&= \omega \log_2(\omega) + \omega \log_2(\log_2 (\omega)) + 6\omega,
	\end{align*}
	and furthermore
	\begin{align*}
	\chi(P',\preceq', \psi')
	&\le \max\left\{\chi(P,\preceq, \psi), \  d_{(P,\preceq, \psi)}(K) + \omega  \lceil \log_2(|Q|+1) \rceil   \right\} \\
	&\le \max\left\{\chi(P,\preceq, \psi), \  2\omega \log_2 (\omega) +2\omega \log_2(\log_2 (\omega)) + 10\omega   \right\}.
	\end{align*}
	
	Hence $(P',\preceq', \psi')$ satisfies the inductive hypothesis. Since $q^*\in Q\subseteq P'$, the interval $I$ is coloured by $\psi'$. As $I$ is not coloured by $\psi$, the number of intervals of $\mathcal{I}$ that are not coloured by $\psi'$ is strictly less than the number of intervals of $\mathcal{I}$ that are not coloured by $\psi$. Hence by induction there exists a complete pillar assignment $(P^*,\preceq^*, \psi^*)$ extending $(P',\preceq', \psi')$ (and thus $(P,\preceq, \psi)$) with $\chi(P^*,\preceq^*, \psi^*) \le 2\omega \log_2 (\omega) +2\omega \log_2(\log_2 (\omega)) + 10\omega$ as required.
\end{proof}

We remark that with more careful arguments it is possible to improve the lower order terms slightly, but we are not aware of a way to improve the leading constant.

To prove that $K_4$-free circle graphs are 19-colourable, one should prove the following modification of Theorem~\ref{induction}. We believe that 19 is still far from optimal and that more specialized arguments could provide significant improvements to this bound. So we only sketch the proof and in particular the required modifications to the proof of Theorem~\ref{induction}.

\begin{theorem}\label{K4}
	Let $G$ be a circle graph with clique number at most $3$, and let $\mathcal{I}$ be an interval system with overlap graph $G$.
	Let $(P,\preceq, \psi)$ be a pillar assignment of $\mathcal{I}$ such that $\chi(P,\preceq, \psi) \le 19$, and $d(K)_{(P,\preceq, \psi)}\le 13$ for every arch $K$ of $(P,\preceq, \psi)$.
	Then there is a complete pillar assignment $(P^*,\preceq^*, \psi^*)$ of $\mathcal{I}$ extending $(P,\preceq, \psi)$ with ${\chi(P^*,\preceq^*, \psi^*)} \le 19$.
\end{theorem}

\begin{proof}[Sketch of proof.]
	As before we proceed by induction on the number of intervals that are not coloured by $\psi=\psi_{(P,\preceq)}$.
	Let $K$ be an arch of $(P,\preceq, \psi)$ that contains some interval of $\mathcal{I}$.
	
	The pillars $q_1,\ldots ,q_n $ as in the proof of Theorem~\ref{induction} are then chosen so that $d_{(P,\preceq, \psi)}(q_{i-1}, q_i) = 7$. By Lemma~\ref{extremal} it then follows that $n\le 7$. Additionally each pillar $q_i$ can be chosen so that there is an interval $I_i\in \mathcal{I}$ coloured by $\psi$ that overlaps with every interval of $\mathcal{I}$ that is contained in $K$ and contains $q_i$.
	Then we let $(P',\preceq', \psi')$ be the pillar assignment extending $(P,\preceq, \psi)$ with $P'= P\cup \{q_1,\ldots ,q_n \}$ as in Lemma~\ref{devide and conquer}.
	We can observe that $\omega(\mathcal{I}_{q_i})\le 2$ since every interval of $\mathcal{I}_{q_i}$ overlaps with $I_i$. This observation provides a slight improvement to the bounds obtained from Lemma~\ref{devide and conquer}, in particular, instead of taking $\omega \le 3$ for the resulting bounds, we may instead replace each occurrence of ``$\omega$" with ``2".
	Then $\chi(P',\preceq', \psi')\le 19$ and $d_{(P',\preceq', \psi')}(K')\le 7+6=13$ for every arch $K'$ of $(P',\preceq', \psi')$ contained in $K$.
	
	We did not pick a pillar $q^*$ at the start to guarantee that $\psi'$ colours an additional interval of $\mathcal{I}$. However if it happens that $ \psi'$ does not colour an additional interval of $\mathcal{I}$, then for each arch $K'$ of $(P',\preceq', \psi')$ contained in $K$, we would have that $d_{(P',\preceq', \psi')}(K')\le 7$. So in this case we may simply extend $(P',\preceq', \psi')$ by a single such pillar $q^*$ since some such arch $K'$ must contain an interval of $\mathcal{I}$.
	The existence of such a complete pillar assignment $(P^*,\preceq^*, \psi^*)$ then follows by induction.
\end{proof}

\section{Lower bound}

In this section we give a simple construction to prove Theorem~\ref{lower}. We find it more convenient to use a chord diagram representation of our circle graphs, rather than the interval overlap representations that were used to prove Theorem~\ref{main} in the previous sections. We allow chords to coincide and consider the chords to be open, so two chords that share an endpoint only intersect if they share both their endpoints.
It can easily be shown that circle graphs are exactly intersection graphs of open chords on a circle where chords can coincide.

The construction is inspired by those given by Kostochka~\cite{kostochka1988upper} for both circle graphs and their complements, as well his proof that the complements of circle graphs are $\chi$-bounded.
With essentially the same arguments, our construction also yields a new proof that there are complements of circle graphs with clique number at most $\omega$ and chromatic number at least $\omega(\ln \omega -O(1))$.

For positive integers $\omega$ and $n$ with $n>3\omega-3$ we define a chord diagram $\mathcal{D}_{n,\omega}$ as follows.
Let $p_1, q_1, p_2, q_2, \dots , p_n , q_n$ be points on a circle in cyclic clockwise order.
Now for each $i\in [n]$, and $j\in [\omega-1]$, let $\mathcal{C}_{i,j}$ consist of exactly $\left\lfloor \frac{\omega}{j+1} \right\rfloor $ coinciding open chords with endpoints $p_i, q_{i+j}$ (taking $i+j$ modulo $n$). Then let $\mathcal{D}_{n,\omega}=\bigcup_{i\in [n]} \bigcup_{j\in [\omega -1]} \mathcal{C}_{i,j}$. For an example, see Figure~\ref{fig:chord}, which illustrates the chord diagram $\mathcal{D}_{17,6}$.

\begin{figure}
\centering 
\begin{tikzpicture}
\draw[thick] (0,0) circle (4cm);
\foreach \i in {1,...,17} {
	\coordinate (N\i) at (\i*360/17 - 4.5:4cm);
	\fill[black] (N\i) circle (0.065 cm);
	\coordinate (M\i) at (\i*360/17 + 4.4:4cm);
	\fill[black] (M\i) circle (0.065 cm);
}
\foreach \i in {1,...,16} {
	\pgfmathparse{\i + 1}
	\edef\j{\pgfmathresult}
	\draw[ultra thick] (N\i) -- (M\j);
}
\foreach \i in {17} {
	\pgfmathparse{\i + 1 - 17}
	\edef\j{\pgfmathresult}
	\draw[ultra thick] (N\i) -- (M\j);
}
\foreach \i in {1,...,15} {
	\pgfmathparse{\i + 2}
	\edef\j{\pgfmathresult}
	\draw[thick] (N\i) -- (M\j);
}
\foreach \i in {16,17} {
	\pgfmathparse{\i + 2 - 17}
	\edef\j{\pgfmathresult}
	\draw[thick] (N\i) -- (M\j);
}
\foreach \i in {1,...,14} {
	\pgfmathparse{\i + 3}
	\edef\j{\pgfmathresult}
	\draw[thin] (N\i) -- (M\j);
}
\foreach \i in {15,16,17} {
	\pgfmathparse{\i + 3 - 17}
	\edef\j{\pgfmathresult}
	\draw[thin]  (N\i) -- (M\j);
}
\foreach \i in {1,...,13} {
	\pgfmathparse{\i + 4}
	\edef\j{\pgfmathresult}
	\draw[thin]  (N\i) -- (M\j);
}
\foreach \i in {14,15,16,17} {
	\pgfmathparse{\i + 4 - 17}
	\edef\j{\pgfmathresult}
	\draw[thin]  (N\i) -- (M\j);
}
\foreach \i in {1,...,12} {
	\pgfmathparse{\i + 5}
	\edef\j{\pgfmathresult}
	\draw[thin]  (N\i) -- (M\j);
}
\foreach \i in {13,14,15,16,17} {
	\pgfmathparse{\i + 5 - 17}
	\edef\j{\pgfmathresult}
	\draw[thin]  (N\i) -- (M\j);
}
\end{tikzpicture}
\caption{The chord diagram of $\mathcal{D}_{17,6}$. The thickness of the chords corresponds to the number of chords that coincide.}
\label{fig:chord}
\end{figure}
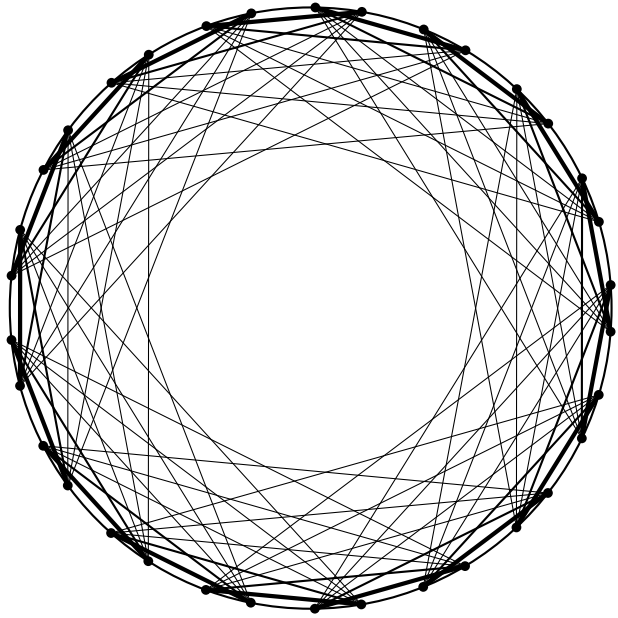

Now we prove some bounds on the number of chords contained in $\mathcal{D}_{n,\omega}$, as well as the size of the largest set of pairwise intersecting and pairwise disjoint chords in $\mathcal{D}_{n,\omega}$.

\begin{lemma}\label{vertices}
	The set of chords $\mathcal{D}_{n,\omega}$ has size greater than $n\omega (\ln \omega - 2)$.
\end{lemma}

\begin{proof}
	By definition,
	\[
	|\mathcal{D}_{n,\omega}| = 
	\left|\bigcup_{i\in [n]} \bigcup_{j\in [\omega -1]} \mathcal{C}_{i,j} \right| =
	\sum_{i=1}^n \sum_{j=1}^{\omega -1} |\mathcal{C}_{i,j}|
	= n \sum_{j=1}^{\omega -1} \left\lfloor \frac{\omega}{j+1} \right\rfloor
	= n \sum_{j=1}^{\omega} \left\lfloor \frac{\omega}{j} \right\rfloor - n\omega.
	\]
	Then observe that:
	\[
	n \sum_{j=1}^{\omega} \left\lfloor \frac{\omega}{j} \right\rfloor - n\omega
	\ge  n \sum_{j=1}^{\omega} \frac{\omega}{j} - 2n\omega
	= n\omega \sum_{j=1}^{\omega} \frac{1}{j} - 2n\omega
	> n\omega \ln \omega - 2n\omega 
	= n\omega (\ln \omega - 2).
	\]
	Hence the lemma follows.
\end{proof}

\begin{lemma}\label{clique}
	There are no $\omega+1$ pairwise intersecting chords contained in $\mathcal{D}_{n,\omega}$.
\end{lemma}

\begin{proof}
	It is enough to show that if $C\subseteq \mathcal{D}_{n,\omega}$ is a collection of pairwise intersecting chords, then $|C|\le \omega$.
	Let $P$ be the set of endpoints of chords in $C$ that are contained in $\{p_1,\dots p_n\}$, and similarly for $Q$. Then $|P|=|Q|$ as a pair of open chords that share an endpoint only intersect if they share both their endpoints.
	Furthermore after possibly rotating the chords of $C$ around the circle, we can assume without loss of generality that $P=\{p_{a_1},\dots ,p_{a_\ell}\}$ and $Q=\{q_{b_1},\dots ,q_{b_\ell}\}$ with $a_1<\dots < a_\ell \le b_1 < \dots < b_\ell$, and that every chord of $C$ has one of $\{p_{a_1},q_{b_1}\}, \dots , \{p_{a_\ell},q_{b_\ell}\}$ as its endpoints. For each $i\in [\ell]$, there are exactly $\left\lfloor \frac{\omega}{(b_i - a_i -1) + 1}    \right\rfloor= \left\lfloor \frac{\omega}{b_i - a_i}    \right\rfloor$
	chords with endpoints contained in $\{p_{a_1},q_{b_1}\}$.
	Therefore
	\[
	|C|\le \sum_{i=1}^{\ell} \left\lfloor \frac{\omega}{b_i - a_i}    \right\rfloor
	\le \sum_{i=1}^{\ell} \frac{\omega}{b_i - a_i}
	\le \sum_{i=1}^{\ell} \frac{\omega}{\ell}
	=\omega.\]
\end{proof}

\begin{lemma}\label{stable}
	There is no set of $n$ pairwise disjoint chords contained in $\mathcal{D}_{n,\omega}$.
\end{lemma}

\begin{proof}
	Let $S$ be a set of pairwise disjoint chords of $\mathcal{D}_{n,\omega}$. Now consider an auxiliary directed graph $G$ on vertex set $\{v_1,\dots, v_n \}$ where there is an edge directed from $v_i$ to $v_j$ whenever $S$ contains a chord with endpoints $p_i$ and $q_j$. First note that $|S|=|E(G)|$ as all the chords with endpoints $p_i$ and $q_j$ intersect.
	
	Now observe that $G$ is outerplanar, with the natural embedding of $v_1,\dots, v_n$ being on the circle in clockwise order and all edges of $G$ directed in the clockwise direction.
	In a directed outerplanar graph with such an embedding, all cycles contain a directed path of length 2 in the clockwise direction.
	However $G$ has no directed path of length 2 as such a path with internal vertex $v_i$ would imply that $S$ contains a chord with an endpoint $p_i$, and another with the endpoint $q_i$, a contradiction since all such chords of $\mathcal{D}_{n,\omega}$ intersect.
	Hence $G$ is a forest, and so $|S|=|E(G)| < |V(G)| = n$ as required.
\end{proof}

We now prove Theorem~\ref{lower}, that for every positive integer $k$ there is a circle graph with clique number at most $k$ and chromatic number at least $k(\ln k - 2)$.

\begin{proof}[Proof of Theorem~\ref{lower}]
	For a positive integer $\omega$, choose some $n>3\omega -3$. Let $G_{n,\omega}$ be the intersection graph of the chord diagram $\mathcal{D}_{n,\omega}$, so $G_{n,\omega}$ is a circle graph.
	By Lemma~\ref{clique}, the graph $G_{n,\omega}$ has clique number at most $\omega$.
	By Lemma~\ref{vertices}, $|V(G_{n,\omega})|>n\omega (\ln \omega - 2)$, and by Lemma~\ref{stable}, the stable sets of $G_{n,\omega}$ all have size less than $n$.
	Hence $\chi(G_{n,\omega}) > \frac{n\omega (\ln \omega - 2)}{n} = \omega (\ln \omega - 2)$ as desired.
\end{proof}

\section*{Acknowledgements}

The author thanks the anonymous referee for helpful comments.

\bibliographystyle{amsplain}
%\bibliography{klogk.bib}

\begin{thebibliography}{10}
	
	\bibitem{ageev1996triangle}
	Alexander~A Ageev, \emph{A triangle-free circle graph with chromatic number 5},
	Discrete Mathematics \textbf{152} (1996), no.~1-3, 295--298.
	
	\bibitem{bar1995vassiliev}
	Dror Bar-Natan, \emph{On the Vassiliev knot invariants}, Topology \textbf{34}
	(1995), no.~2, 423--472.
	
	\bibitem{birman1993knot}
	Joan~S Birman and Xiao-Song Lin, \emph{Knot polynomials and Vassiliev's
		invariants}, Inventiones mathematicae \textbf{111} (1993), no.~1, 225--270.
	
	\bibitem{bouchet1987unimodularity}
	Andr{\'e} Bouchet, \emph{Unimodularity and circle graphs}, Discrete mathematics
	\textbf{66} (1987), no.~1-2, 203--208.
	
	\bibitem{bravyi2007measurement}
	Sergey Bravyi and Robert Raussendorf, \emph{Measurement-based quantum
		computation with the toric code states}, Physical Review A \textbf{76}
	(2007), no.~2, 022304.
	
	\bibitem{capoyleas1992turan}
	Vasilis Capoyleas and J{\'a}nos Pach, \emph{A Tur{\'a}n-type theorem on chords
		of a convex polygon}, Journal of Combinatorial Theory, Series B \textbf{56}
	(1992), no.~1, 9--15.
	
	\bibitem{chudnovsky2016induced}
	Maria Chudnovsky, Alex Scott, and Paul Seymour, \emph{Induced subgraphs of
		graphs with large chromatic number. V. Chandeliers and strings}, Journal of
	Combinatorial Theory, Series B \textbf{150} (2021), 195--243.
	
	\bibitem{davies2020vertex}
	James Davies, \emph{Vertex-minor-closed classes are $\chi$-bounded}, arXiv
	preprint arXiv:2008.05069 (2020).
	
	\bibitem{davies2021colouring}
	James Davies, Tomasz Krawczyk, Rose McCarty, and Bartosz Walczak,
	\emph{Grounded $L$-graphs are polynomially $\chi$-bounded},  arXiv
	preprint arXiv:2108.05611 (2021).
	
	\bibitem{daviescircle}
	James Davies and Rose McCarty, \emph{Circle graphs are quadratically
		$\chi$-bounded}, Bulletin of the London Mathematical Society \textbf{53}
	(2021), no.~3, 673--679.
	
	\bibitem{dujmovic2004linear}
	Vida Dujmovi{\'c} and David~R Wood, \emph{On linear layouts of graphs},
	Discrete Mathematics and Theoretical Computer Science \textbf{6} (2004),
	no.~2, 339--358.
	
	\bibitem{erdos1935combinatorial}
	Paul Erd\H{o}s and George Szekeres, \emph{A combinatorial problem in geometry},
	Compositio mathematica \textbf{2} (1935), 463--470.
	
	\bibitem{even1971queues}
	Shimon Even and Alon Itai, \emph{Queues, stacks and graphs}, Theory of Machines
	and Computations, Elsevier, 1971, pp.~71--86.
	
	\bibitem{flajolet1980sequence}
	Philippe Flajolet, Jean Fran{\c{c}}on, and Jean Vuillemin, \emph{Sequence of
		operations analysis for dynamic data structures}, Journal of Algorithms
	\textbf{1} (1980), no.~2, 111--141.
	
	\bibitem{de1984characterization}
	Hubert de~Fraysseix, \emph{A characterization of circle graphs}, European
	Journal of Combinatorics \textbf{5} (1984), no.~3, 223--238.
	
	\bibitem{garey1980complexity}
	Michael~R Garey, David~S Johnson, Gary~L Miller, and Christos~H Papadimitriou,
	\emph{The complexity of coloring circular arcs and chords}, SIAM Journal on
	Algebraic Discrete Methods \textbf{1} (1980), no.~2, 216--227.
	
	\bibitem{gavril1973}
	Fanica Gavril, \emph{Algorithms for a maximum clique and a maximum independent set of a circle graph}, Networks \textbf{3} (1973), no.~3, 261--273.
	
	\bibitem{geelen2020grid}
	Jim Geelen, O-joung Kwon, Rose McCarty, and Paul Wollan, \emph{The grid theorem
		for vertex-minors}, Journal of Combinatorial Theory, Series B, doi:10.1016/j.jctb.2020.08.004 (2020).
	
	\bibitem{golumbic2004algorithmic}
	Martin~C Golumbic, \emph{Algorithmic graph theory and perfect graphs},
	Elsevier, 2004.
	
	\bibitem{gyarfas1985chromatic}
	Andr{\'a}s Gy{\'a}rf{\'a}s, \emph{On the chromatic number of multiple interval
		graphs and overlap graphs}, Discrete mathematics \textbf{55} (1985), no.~2,
	161--166.
	
	\bibitem{gyarfas1987problems}
	Andr{\'a}s Gy{\'a}rf{\'a}s, \emph{Problems from the world surrounding perfect graphs},
	Applicationes Mathematicae \textbf{3} (1987), no.~19, 413--441.
	
	\bibitem{gyarfas1985covering}
	Andr{\'a}s Gy{\'a}rf{\'a}s and Jen{\"o} Lehel, \emph{Covering and coloring
		problems for relatives of intervals}, Discrete Mathematics \textbf{55}
	(1985), no.~2, 167--180.
	
	\bibitem{hofacker1998combinatorics}
	Ivo~L Hofacker, Peter Schuster, and Peter~F Stadler, \emph{Combinatorics of RNA
		secondary structures}, Discrete Applied Mathematics \textbf{88} (1998),
	no.~1-3, 207--237.
	
	\bibitem{kostochka1997covering}
	Alexandr Kostochka and Jan Kratochv{\'\i}l, \emph{Covering and coloring
		polygon-circle graphs}, Discrete Mathematics \textbf{163} (1997), no.~1-3,
	299--305.
	
	\bibitem{kostochka1988upper}
	Alexandr Kostochka, \emph{Upper bounds on the chromatic number of graphs},
	Trudy Inst. Mat.(Novosibirsk) \textbf{10} (1988), no.~Modeli i Metody Optim.,
	204--226.
	
	\bibitem{krawczyk2017line}
	Tomasz Krawczyk and Bartosz Walczak, \emph{On-line approach to off-line
		coloring problems on graphs with geometric representations}, Combinatorica
	\textbf{37} (2017), no.~6, 1139--1179.
	
	\bibitem{marie2013chord}
	Nicolas Marie and Karen Yeats, \emph{A chord diagram expansion coming from some
		dyson-schwinger equations}, Communications in Number Theory and Physics
	\textbf{7} (2013), no.~2, 251--291.
	
	\bibitem{nenashev2012upper}
	Gleb~V Nenashev, \emph{An upper bound on the chromatic number of a circle graph
		without $K_4$}, Journal of Mathematical Sciences \textbf{184} (2012), no.~5,
	629--633.
	
	\bibitem{rok2019coloring}
	Alexandre Rok and Bartosz Walczak, \emph{Coloring curves that cross a fixed
		curve}, Discrete \& Computational Geometry \textbf{61} (2019), no.~4,
	830--851.
	
	\bibitem{scott2020induced}
	Alex Scott and Paul Seymour, \emph{Induced subgraphs of graphs with large
		chromatic number. VI. Banana trees}, Journal of Combinatorial Theory, Series
	B \textbf{145} (2020), 487--510.
	
	\bibitem{scott2020survey}
	Alex Scott and Paul Seymour, \emph{A survey of $\chi$-boundedness}, Journal of Graph Theory
	\textbf{95} (2020), no.~3, 473--504.
	
	\bibitem{sherwani2012algorithms}
	Naveed~A Sherwani, \emph{Algorithms for VLSI physical design automation},
	Springer Science \& Business Media, 2012.
	
	\bibitem{thomassen1994every}
	Carsten Thomassen, \emph{Every planar graph is 5-choosable}, Journal of
	Combinatorial Theory, Series B \textbf{62} (1994), no.~1, 180--181.
	
	\bibitem{touchard1952probleme}
	Jacques Touchard, \emph{Sur un probleme de configurations et sur les fractions
		continues}, Canadian Journal of Mathematics \textbf{4} (1952), 2--25.
	
	\bibitem{van2004graphical}
	Maarten Van~den Nest, Jeroen Dehaene, and Bart De~Moor, \emph{Graphical
		description of the action of local Clifford transformations on graph states},
	Physical Review A \textbf{69} (2004), no.~2, 022316.
	
\end{thebibliography}

\end{document}